\documentclass[12pt,a4paper]{amsart}
\usepackage{amsfonts}
\usepackage{amsthm}
\usepackage{amsmath}
\usepackage{amscd}
\usepackage{amssymb}
\usepackage[latin2]{inputenc}
\usepackage{t1enc}
\usepackage[mathscr]{eucal}
\usepackage{indentfirst}
\usepackage{graphicx}
\usepackage{graphics}
\usepackage{pict2e}
\usepackage{epic}
\numberwithin{equation}{section}
\usepackage[margin=2.9cm]{geometry}
\usepackage{epstopdf} 
\usepackage{enumerate}
\usepackage[shortlabels]{enumitem}
\usepackage{verbatim}
\usepackage{cite}
\usepackage{mathtools}
\usepackage{mathrsfs}
\usepackage{xcolor}
\usepackage{hyperref}
\usepackage{soul}
\usepackage[normalem]{ulem}

\theoremstyle{plain}
\newtheorem{Th}{Theorem}[section]
\newtheorem{Lemma}[Th]{Lemma}
\newtheorem{Cor}[Th]{Corollary}

\newtheoremstyle{named}{}{}{\itshape}{}{\bfseries}{.}{.5em}{\thmnote{#3}}
\theoremstyle{named}

\theoremstyle{definition}
\newtheorem{Def}[Th]{Definition}

\newcommand{\vol}{{\rm{vol}}}
\newcommand{\tr}{{\rm{tr}}}

\newcommand{\ww}{\mathtt{w}}
\newcommand{\dvol}{\operatorname{dvol}}
\definecolor{ao(english)}{rgb}{0.0, 0.5, 0.0}
\renewcommand\Re{\operatorname{Re}}

\begin{document}

\title[Volume comparison on finite-volume hyperbolic 3-manifolds]{Volume comparison on finite-volume hyperbolic 3-manifolds}

\author{Ruojing Jiang}

\address{Massachusetts Institute of Technology, Department of Mathematics, Cambridge, MA 02139} 
\email{ruojingj@mit.edu}

\author{Franco Vargas Pallete}
\address{School of Mathematical and Statistical Sciences, Arizona State University, Tempe, AZ 85287}
\email{fevargas@asu.edu}

 \subjclass[2020]{} 
 \date{}

\begin{abstract}  
On finite-volume hyperbolic $3$-manifolds, we compare volumes of different metrics using the exponential convergence of Ricci-DeTurck flow toward the hyperbolic metric $h_0$. We prove that among metrics with scalar curvature bounded below by $-6$, $h_0$ minimizes the volume. Moreover, for metrics that are either uniformly $C^2$-close to $h_0$ or asymptotically cusped of order at least two, equality holds if and only if the metric is isometric to $h_0$.
\end{abstract}

\maketitle

\section{Introduction}

This paper focuses on the applications of Ricci flow on hyperbolic $3$-manifolds of finite volume.
Specifically, we compare the volume of $M$ with respect to different metrics. A fundamental conjecture attributed to Schoen \cite{Schoen1989} posits that, on a closed $n$-manifold that admits a hyperbolic metric, the hyperbolic metric minimizes volume among all metrics with scalar curvature bounded below by $-n(n-1)$. In dimension three, this conjecture was resolved by Perelman as a consequence of his work on Ricci flow with surgery and the proof of the Geometrization Conjecture.

Building on this perspective, the Ricci flow has been employed as a powerful tool for deriving geometric and topological inequalities. Agol, Storm, and Thurston \cite{Agol-Storm-Thurston} used Ricci flow to establish volume comparison results for compact $3$-manifolds with boundary consisting of minimal surfaces. Their approach involves doubling the manifold and applying Perelman's techniques to the resulting closed manifolds. In higher dimensions, Hu, Ji, and Shi \cite{HuJiShi2020} investigated the volume comparison in the setting of strictly stable conformally compact Einstein manifolds. By analyzing the exponential convergence rate to Einstein metrics, they established volume minimizing properties for such metrics in dimensions $n\geq 4$.

On a hyperbolic $3$-manifold of finite volume, to obtain the volume comparison between different metrics, we use the Ricci flow with a specific version of surgery on cusped manifolds introduced by Bessi{\`e}res, Besson, and Maillot \cite{Bessieres-Besson-Maillot}. It is called \emph{Ricci flow with bubbling-off}, with assumption that the initial metric has a cusp-like structure. Their work indicates that, after a finite number of surgeries, the solution converges smoothly to the hyperbolic metric on compact sets. However, this convergence may fail to extend globally on $M$, since the cuspidal ends allow for nontrivial Einstein variations that can alter the asymptotic behavior. On the other hand, Bamler \cite{Bamler} showed that if the initial metric is a small $C^0$ perturbation of the hyperbolic metric, then the Ricci flow converges on compact sets and remains asymptotic to the same hyperbolic structure for all time. 

In \cite{Jiang-VargasPallete_RF}, the authors provided a more quantitative version of the stability of cusped hyperbolic manifolds under normalized Ricci-DeTurck flow. We impose additional conditions on the initial metric and use Bamler's stability result \cite{Bamler} to deal with trivial Einstein variations. The strategy uses maximal regularity theory and interpolation techniques, following the approach of Angenent~\cite{Angenent_1990}, which extends the work of Da Prato and Grisvard \cite{DaPratoGrisvard1975}. By working with a pair of densely embedded Banach spaces and an operator that generates a strongly continuous analytic semigroup, we obtain maximal regularity for solutions of the normalized Ricci-DeTurck flow. This framework enables us to derive exponential convergence to the hyperbolic metric, with optimal decay rate given by the spectral estimate of the linearized operator.

\subsection{Main results}
On a finite-volume hyperbolic $3$-manifold, we showed in \cite{Jiang-VargasPallete_RF} that if the initial metric $h$ is $C^0$-close to the hyperbolic metric $h_0$, then the normalized Ricci-DeTurck flow exists for all time and converges exponentially fast to $h_0$ in a weighted H{\"o}lder norm for $t\geq 1$. In Theorem~\ref{thm_ricci_flow} below, we establish a new version that yields the convergence rate for all $t \geq 0$, under the stronger assumption that $h$ is $C^2$-close to $h_0$.
This attractivity result further leads to a comparison of the volume of $M$ with respect to different metrics. 

To introduce the theorem, we need the following definition. 

\renewcommand{\theTh}{\thesection.\arabic{Th}}
\begin{Def}\label{def_asymptotically_cusp}
A Riemannian metric $h$ on $M$ is said to be \emph{asymptotically cusped of order $k$} if there exist a constant $\lambda>0$ and a hyperbolic metric $h_{cusp}$ defined on the cusp $\mathcal{C}=\cup_iT_i\times [0,\infty)$, such that $\lambda h|_{\mathcal{C}}-h_{cusp}$ tends to zero at infinity in $C^k$.
\end{Def}

\begin{Th}\label{thm_volume_comparison}
    Let $(M,h_0)$ be a hyperbolic $3$-manifold of finite volume, and let $h$ be a Riemannian metric on $M$ with scalar curvature $R(h)\geq -6$. Then \begin{equation*}
        \vol_h(M)\geq \vol_{h_0}(M).
    \end{equation*} 
    Furthermore, suppose that $h$ either satisfies $\Vert h-h_0\Vert _{C^2(M)}\leq \epsilon$ for a given constant $\epsilon>0$, or it is asymptotically cusped of order at least two. Then the equality holds if and only if $h$ is isometric to $h_0$.
\end{Th}

\subsection{Organization}
The paper is organized as follows. Section~\ref{section_bkg_rf} reviews the background on Ricci and Ricci-DeTurck flow and establishes the stability of the hyperbolic metric under $C^k$ perturbations. In Section~\ref{section_rf_conv}, we apply this stability result to derive exponential decay estimates toward the hyperbolic metric for all time, which are then used in the proof of Theorem~\ref{thm_volume_comparison}. Section~\ref{section_proof} contains the proof of Theorem~\ref{thm_volume_comparison}, and Section~\ref{section_applications} presents some brief applications.

\section*{Acknowledgements}
FVP thanks IHES for their hospitality during a phase of this work. FVP was partially funded by European Union (ERC, RaConTeich, 101116694)\footnote{Views and opinions expressed are however those of the author(s) only and do not necessarily reflect those of the European Union or the European Research Council Executive Agency. Neither the European Union nor the granting authority can be held responsible for them.}

\section{Background of Ricci flow}\label{section_bkg_rf}

In this section, we will briefly review the tools used to prove Theorem~\ref{thm_volume_comparison}.
\subsection{Normalized Ricci flow and Ricci-DeTurck flow}
The \emph{normalized Ricci flow} on $M$ is defined as \begin{equation}\label{RF}
    \frac{\partial h}{\partial t}=-2Ric(h)-4h.
\end{equation}
However, this evolution equation is only weakly parabolic. To achieve strict parabolicity, we introduce the following DeTurck-modified version. 
The \emph{normalized Ricci-DeTurck flow} for \eqref{RF} is given by \begin{equation}\label{DRF}
    \frac{\partial h}{\partial t}=-2Ric(h)-4h+\nabla_iV_j+\nabla_jV_i,
\end{equation}
where \begin{equation}\label{eq_V}
    V_j=h_{jk}h^{pq}\left(\Gamma_{pq}^k-(\Gamma_{h_0})_{pq}^k\right).
\end{equation}
Moreover, there is a family of diffeomeorphisms $\Phi(t):M\to M$ which solves 
$$
\begin{cases}
    \dfrac{\partial}{\partial t}\Phi(t)=-V(\Phi(t),t),\\
    \Phi(0)=Id,
\end{cases}
$$
where the components of $V$ is defined by $V^j=h^{jk}V_k$. If $h(t)$ solves \eqref{DRF}, then $\Phi(t)^*h(t)$ is a solution to \eqref{RF}.

\subsection{Ricci flow with bubbling-off}\label{subsection_bubbling}

In this section, we review the notion of Ricci flow with bubbling-off. For more details, readers are encouraged to consult the book by Bessi{\`e}res, Besson, Boileau, Maillot, and Porti \cite{BBB+10}. 

The construction of Ricci flow with this specific version of surgery on the cusped manifold $M$ was established by Bessi{\`e}res, Besson, and Maillot in \cite{Bessieres-Besson-Maillot}, under the assumption that the initial metric $h$ admits a cusp-like structure (Definition~\ref{def_cusplike}). This means that the restriction of $h$ on each cusp $T_j\times [0,\infty)$ is asymptotic to a hyperbolic metric $h_{cusp}=e^{-2s}h_{T_j}+ds^2$ in the cuspidal end.
Note that the hyperbolic metric $h_{cusp}$ is not unique, it varies based on different choices of flat metrics $h_{T_j}$ on $T_j$. The cusp-like structure ensures that the universal cover $(B^3,h)$ has bounded geometry, allowing the existence theorem of Ricci flow with surgery (Theorem 2.17, \cite{Bessieres-Besson-Maillot}) to apply, and thus making it possible to consider an equivalent version that passes to the quotient (Addendum 2.19, \cite{Bessieres-Besson-Maillot}).

Furthermore, their work examines the long-time behavior of the Ricci flow on $M$ starting from a metric $h(0)$ with a cusp-like structure. After a finite number of surgeries, as $t$ goes to infinity, the solution $h(t)$ converges smoothly to the hyperbolic metric $h_0$ on balls of radius $R$ for all $R>0$ (Theorem 1.2 of \cite{Bessieres-Besson-Maillot}). However, as indicated in the stability theorem (see Theorem~\ref{thm_stability_cusp} below), outside these balls, the cusp-like structure of $h(0)$ is preserved for all time. Therefore, if $h(0)$ is asymptotic to some $h_{cusp}$ different from the restriction of $h_0$ on the cusp, then the convergence cannot be global on $M$. 

It is worth noting that the proof of the stability theorem relies on a different construction of surgery. Since $M$ is both irreducible and lacks finite quotients of $S^3$ or $S^2\times S^1$, any surgery in $M$ splits off a 3-sphere and does not change the topology, the authors focused only on metric surgeries that change the metric on some 3-balls. This version of surgery is called \emph{Ricci flow with bubbling-off} (Definition 5.2.8, \cite{BBB+10}). The main distinction from the usual Hamilton-Perelman surgery is that, the bubbling-off occurs before a singularity appears. Moreover, in addition to the surgery parameters $r$ and $\delta$, they introduced new \emph{associated cutoff parameters} $H$ and $\Theta$ to determine when the scalar curvature at one end of a neck is large enough to perform a bubbling-off. In particular, this construction of bubbling-off is essential in proving the stability of cusp-like structures at infinity.

\begin{Def}[Cusp-like metrics]\label{def_cusplike}
    A metric $h$ on $M$ admits a \emph{cusp-like structure} if it is asymptotically cusped of order $k$ for any integer $k$. In other words, there exists a hyperbolic metric $h_{cusp}$ on the cusp and $\lambda>0$, such that $\lambda h-h_{cusp}$ approaches zero at infinity in the $C^k$-norm for each integer $k$.
\end{Def}

\begin{Th}[Stability of cusp-like structures (Theorem 2.22, \cite{Bessieres-Besson-Maillot})]\label{thm_stability_cusp}
Let $h(0)$ be a cusp-like metric on $M$. Then there exists a normalized Ricci flow with bubbling-off $h(t)$ on $M$, defined for all $t\in [0,\infty)$, starting at $h(0)$.

    Moreover, there is a factor $\lambda(t)>0$, such that $\lambda(t)h(t)-h_{cusp}$ goes to zero at infinity in the cuspidal end, in $C^k$-norm for each integer $k$, uniformly for $t\in [0,\infty)$. This means that $h(t)$ remains asymptotic to the same hyperbolic metric on the cusp for all time. 
\end{Th}

In \cite{Jiang-VargasPallete_entropy}, we generalized the theorem to asymptotically cusped metrics of any order $k\geq 2$. 

\begin{Th}[Stability of asymptotically cusped metrics]\label{thm_stability_cusp_asymp}
Let $h(0)$ be an asymptotically cusped metric on $M$ of order $k\geq 2$. 
Then there exists a normalized Ricci flow with bubbling-off $h(t)$ on $M$, defined for all $t\in [0,\infty)$, starting at $h(0)$.

    Moreover, assume that $\Vert Rm(h(0))\Vert_{C^{k-1}(M)}<\infty$. Then there is a factor $\lambda(t)>0$, such that $\lambda(t)h(t)-h_{cusp}$ goes to zero at infinity in the cuspidal end in $C^k$ uniformly for $t\in [0,\infty)$. 
\end{Th}

\subsection{Stability for Ricci-DeTurck flow}
In this section, we introduce some stability results associated with the normalized Ricci-DeTurck flow \eqref{DRF}. 

\begin{Lemma}\label{lemma_derivative_estimates}
    Let $(M,h_0)$ be a complete $3$-manifold, and let $\epsilon>0$ be a sufficiently small constant. Given any $k\in\mathbb{N}$ and $C_0>0$, there exists a constant $C=(\epsilon,k,C_0)>0$ such that the following statement holds.  Consider a normalized Ricci-DeTurck flow $g(t)$ defined on $M\times [0,T]$, where $T=T(\epsilon)$ is given by the short-time existence, such that 
    \begin{equation*}
        \Vert g(0)-h_0\Vert_{C^0(M)}<\epsilon,
    \end{equation*}
   and 
    \begin{equation*}
        \Vert g(0)-h_0\Vert_{C^k(M)}\leq C_0.
    \end{equation*} 
    Then \begin{enumerate}
        \item $$\Vert g(t)-h_0\Vert_{C^k(M)}\leq C\quad \forall t\in [0,T],$$
        \item $$\Vert \nabla^{k+1}_{h_0}g(t)\Vert_{C^0(M)}\leq Ct^{-\frac12}\quad \forall t\in [0,T].$$
    \end{enumerate}
\end{Lemma}

\begin{proof}
When $k=1$, the result was established by Simon in \cite[Lemma 2.1]{Simon2005DeformingLipschitz}. We will prove the general case by induction on $k$ following the approach in \cite[Lemma 4.2]{Simon_DeformationO}. In the following proof, $\nabla_{h_0}$ and $|\cdot|_{h_0}$ are always with respect to $h_0$, and we will simplify the notations as $\nabla$ and $|\cdot|$.

We start by proving (1).
Assume that we know already \begin{equation*}
    \Vert g(t)-h_0\Vert_{C^{k-1}(M)}\leq c_0\quad \forall t\in [0,T].
\end{equation*}
Moreover, for some small $\epsilon_0>0$ to be chosen later make $\epsilon$ smaller if needed so that
\begin{equation*}
    \Vert g(t)-h_0\Vert_{C^{0}(M)}\leq \epsilon_0\quad \forall t\in [0,T].
\end{equation*}
According to \cite[Lemma 4.2]{Simon_DeformationO}, we have \begin{align}\label{equ_derivative_lemma_4}
    \frac{\partial}{\partial t}\nabla^kg=& g^{ij}\nabla_i\nabla_j(\nabla^kg)+\sum_{i+j+m=k\,i,j,m\leq k}\nabla^ig^{-1}\ast \nabla^jg^{-1}\ast \nabla^mRm(h_0)\\\nonumber
    &+\sum_{i+j+m+l=k+2,\,i,j,m,l\leq k+1}\nabla^ig^{-1}\ast \nabla^j g^{-1}\ast \nabla^mg\ast \nabla^lg,
\end{align}
and then
\begin{align*}
    \frac{\partial}{\partial t}|\nabla^kg|^2=& g^{ij}\nabla_i\nabla_j|\nabla^kg|^2-2g^{ij}\nabla_i(\nabla^kg)\nabla_j(\nabla^kg)\\
    &+2\sum_{i+j+m=k\,i,j,m\leq k}\nabla^ig^{-1}\ast \nabla^jg^{-1}\ast \nabla^mRm(h_0)\ast\nabla^kg\\
    &+2\sum_{i+j+m+l=k+2,\,i,j,m,l\leq k+1}\nabla^ig^{-1}\ast \nabla^j g^{-1}\ast \nabla^mg\ast \nabla^lg\ast \nabla^kg,
\end{align*}
where $\ast$ represents the tensor product with respect to $h_0$.
By assumption, all lower-order derivatives $\nabla^i g$ with $i\leq k-1$ are bounded by a constant. This implies that 
\begin{align*}
    \frac{\partial}{\partial t}|\nabla^kg|^2\leq & g^{ij}\nabla_i\nabla_j|\nabla^kg|^2-2g^{ij}\nabla_i(\nabla^kg)\nabla_j(\nabla^kg)\\
    &+c_1|\nabla^kg|+c_1|\nabla^kg|^2+c_1|\nabla^{k+1}g||\nabla^kg|\\
    \leq &  g^{ij}\nabla_i\nabla_j|\nabla^kg|^2-2g^{ij}\nabla_i(\nabla^kg)\nabla_j(\nabla^kg)\\
    &+c_2|\nabla^k g|^2+(1-\epsilon_0)|\nabla^{k+1}g|^2+c_2,
\end{align*}
where the last inequality follows from Cauchy-Schwarz inequality, $c_1, c_2$ depend on $h_0,k,\epsilon_0,c_0$. We will omit the dependence on $h_0$ from now on.
Since \begin{equation*}
    2g^{ij}\nabla_i(\nabla^kg)\nabla_j(\nabla^kg)\geq 2(1-\epsilon_0)|\nabla^{k+1}g|^2,
\end{equation*}
substituting it into the previous inequality yields 
\begin{equation}\label{equ_derivative_lemma_1}
    \frac{\partial}{\partial t}|\nabla^kg|^2\leq g^{ij}\nabla_i\nabla_j|\nabla^kg|^2-(1-\epsilon_0)|\nabla^{k+1}g|^2+c_2|\nabla^k g|^2+c_2.
\end{equation}
Similarly, \begin{align}\label{equ_derivative_lemma_2}
    \frac{\partial}{\partial t}|\nabla^{k-1}g|^2\leq & g^{ij}\nabla_i\nabla_j|\nabla^{k-1}g|^2-(1-\epsilon_0)|\nabla^{k}g|^2+c_2|\nabla^{k-1} g|^2+c_2\\\nonumber
    \leq & g^{ij}\nabla_i\nabla_j|\nabla^{k-1}g|^2-(1-\epsilon_0)|\nabla^{k}g|^2+c_3,
\end{align}
where $c_3=c_2c_0^2+c_2$.

Furthermore, we define \begin{equation*}
    \psi(x,t)=\left(a+|\nabla^{k-1}g|^2\right)|\nabla^k g|^2,
\end{equation*}
where $a>0$ is a constant that will be chosen later. Using \eqref{equ_derivative_lemma_1} and \eqref{equ_derivative_lemma_2}, we obtain \begin{align*}
    \frac{\partial}{\partial t}\psi\leq & g^{ij}\nabla_i\nabla_j\psi -(1-\epsilon_0)|\nabla^{k}g|^4+c_3|\nabla^kg|^2\\
    &+ \left(a+|\nabla^{k-1}g|^2\right)\left(-(1-\epsilon_0)|\nabla^{k+1}g|^2+c_2|\nabla^k g|^2+c_2\right)\\
    &-2g^{ij}\nabla_i|\nabla^{k-1}g|^2\nabla_j|\nabla^kg|^2\\
    \leq & g^{ij}\nabla_i\nabla_j\psi -\frac{1}{2}(1-\epsilon_0)|\nabla^{k}g|^4-a(1-\epsilon_0)|\nabla^{k+1}g|^2+c_4\\
    &-2g^{ij}\nabla_i|\nabla^{k-1}g|^2\nabla_j|\nabla^kg|^2,
\end{align*}
where the last inequality follows from Cauchy-Schwarz inequality, and $c_4$ depends on $\epsilon_0,c_2,c_3$. The last term satisfies \begin{align*}
    -2g^{ij}\nabla_i|\nabla^{k-1}g|^2\nabla_j|\nabla^kg|^2\leq & 8(1+\epsilon_0)|\nabla^kg|^2|\nabla^{k-1}g||\nabla^{k+1}g|\\
    \leq & \frac{1}{4}(1-\epsilon_0)|\nabla^kg|^4+16\frac{(1+\epsilon_0)^2}{1-\epsilon_0}c_0^2|\nabla^{k+1}g|^2.
\end{align*}
Hence, \begin{equation*}
    \frac{\partial}{\partial t}\psi(x,t)\leq g^{ij}\nabla_i\nabla_j\psi+\left(16\frac{(1+\epsilon_0)^2}{1-\epsilon_0}c_0^2-a(1-\epsilon_0)\right)|\nabla^{k+1}g|^2-\frac{1}{4}(1-\epsilon_0)|\nabla^{k}g|^4+c_4.
\end{equation*}
Choose $a$ so that $16\frac{(1+\epsilon_0)^2}{1-\epsilon_0}c_0^2-a(1-\epsilon_0)\leq 0$. We obtain $\psi^2\leq (a+c_0^2)^2|\nabla^kg|^4\leq 2a^2|\nabla^kg|^4$. We chose $\epsilon_0$ so that $\epsilon_0\leq \frac12$. Hence it follows that
\begin{equation*}
    \frac{\partial}{\partial t}\psi(x,t)\leq g^{ij}\nabla_i\nabla_j\psi-\frac{1}{4}(1-\epsilon_0)|\nabla^{k}g|^4+c_4\leq g^{ij}\nabla_i\nabla_j\psi-\frac{1}{16a^2}\psi^2+c_4.
\end{equation*}

Next, we cover $M$ by balls with a fixed radius $r>0$, and consider the time independent cut-off function $\eta$ defined in \cite[Lemma 4.1]{Simon_DeformationO} with the following properties: 
\begin{subequations}
\begin{align}
    &\eta(x)=1\quad \forall x\in B_{h_0}(x_0,r),\label{equ_cutoff_1}\\
    &\eta(x)=0\quad \forall x\in M\setminus B_{h_0}(x_0,2r),\label{equ_cutoff_2}\\
    &\eta(x)\in [0,1]\quad \forall x\in M,\label{equ_cutoff_3}\\
    &|\nabla \eta|^2\leq c_5\eta,\label{equ_cutoff_4}\\
    &\nabla_i\nabla_j\eta\geq -c_5,\label{equ_cutoff_5}
\end{align}
\end{subequations}
where the constant $c_5$ depends on $r$. Since $r$ is a fixed number, for example, we may assume $r=1$ and omit the dependence on $r$ from now on.
By \eqref{equ_cutoff_3} and \eqref{equ_cutoff_5}, \begin{align}\label{equ_derivative_lemma_3}
    \frac{\partial}{\partial t}(\psi\eta)\leq & g^{ij}\nabla_i\nabla_j(\psi\eta)-\frac{1}{16a^2}\psi^2\eta-2g^{ij}\nabla_i\psi\nabla_j\eta-\psi g^{ij}\nabla_i\nabla_j\eta+c_4\\\nonumber
    \leq & g^{ij}\nabla_i\nabla_j(\psi\eta)-\frac{1}{16a^2}\psi^2\eta-2g^{ij}\nabla_i\psi\nabla_j\eta+c_5\psi+c_4.
\end{align}
Assume that $(y_0,t_0)$ is an interior point of $B_{h_0}(x_0,2r)\times (0,T)$ where the supremum of $\psi\eta$ along $M\times \lbrace t_0\rbrace$ is attained. Then \begin{align*}
    -2g^{ij}\nabla_i\psi\nabla_j\eta(y_0,t_0) &=-2g^{ij}\frac{1}{\eta}\nabla_i(\psi\eta)\nabla_j\eta(y_0,t_0)+2g^{ij}\frac{\psi}{\eta}|\nabla\eta|^2(y_0,t_0)\\
    &=2g^{ij}\frac{\psi}{\eta}|\nabla\eta|^2(y_0,t_0)\leq 2(1+\epsilon_0)c_5\psi(y_0,t_0),
\end{align*}
where the last inequality applies \eqref{equ_cutoff_4}. When it is combined with \eqref{equ_derivative_lemma_3}, we get \begin{equation*}
    0\leq \frac{\partial}{\partial t}(\psi\eta)(y_0,t_0)\leq -\frac{1}{16a^2}\psi^2\eta(y_0,t_0)+4c_5\psi(y_0,t_0)+c_4.
\end{equation*}
Multiplying by $\eta$ we obtain 
$$\left(\frac{1}{16a^2}\left(\psi\eta\right)^2-4c_5\psi\eta-c_4\eta\right)(y_0,t_0)\leq 0.$$ Therefore by \eqref{equ_cutoff_1}  we conclude that \begin{equation*}
    \psi(x,t)\leq \psi\eta(y_0,t_0)\leq c_6, \quad \forall (x,t)\in B_{h_0}(x,r)\times (0,T),
\end{equation*}
where $c_6$ depends on $a,c_4,c_5$, and therefore on $k,\epsilon_0,c_0$. It shows that \begin{equation*}
    |\nabla^kg|\leq \left(\frac{c_6}{a}\right)^{\frac12}, \quad \forall (x,t)\in M\times (0,T).
\end{equation*}
The bound extends to $t=T$ by continuity, and since $|\nabla^kg(0)|\leq C_0$ by assumption, we obtain a uniform bound on $M\times [0,T]$, which derives (1) by induction.

We now prove (2). Consider the following function $w(x,t)$ defined as $$
w(t)=\begin{cases}
     t\left(C^2+|\nabla^{k}g|^2\right)|\nabla^{k+1} g|^2 & t\in(0,T],\\
   0 &t=0.
\end{cases}
$$
Analogously to the calculation of $\psi$, we can deduce that \begin{align*}
    \frac{\partial}{\partial t}w\leq& g^{ij}\nabla_i\nabla_jw-t\frac{1-\epsilon_0}{2}|\nabla^{k+1}g|^4-tC^2(1-\epsilon_0)|\nabla^{k+2}g|^2+\frac{w}{t}+c_7\\
    \leq &g^{ij}\nabla_i\nabla_jw-t\frac{1-\epsilon_0}{2}|\nabla^{k+1}g|^4+\frac{w}{t}+c_7\\
    =& g^{ij}\nabla_i\nabla_jw-\frac{1-\epsilon_0}{2t}\frac{w^2}{\left(C^2+|\nabla^{k}g|^2\right)^2}+\frac{w}{t}+c_7\\
    \leq & g^{ij}\nabla_i\nabla_jw-\frac{1}{16C^4t}w^2+\frac{w}{t}+c_7,
\end{align*}
where we use the estimate $|\nabla^kg|\leq C$ obtained from (1), and $c_7$ depends on $k,\epsilon_0,C$.

When multiplied by the cut-off function $\eta$, it follows that at the point $(y_0,x_0)$ in the interior of $B_{h_0}(x_0,2r)\times (0,T)$ where the supremum of $w\eta$ is attained, we have \begin{equation*}
   \left( \frac{1}{16C^4t}\left(w\eta\right)^2-\left(\frac{1}{t}+4c_5\right)w\eta-c_7\eta\right)(y_0,t_0)\leq 0.
\end{equation*}
As before, $w$ is then uniformly bounded above by a constant on $M\times (0,T)$, which extends to $M\times [0,T]$ by continuity.
This implies that $|\nabla^{k+1}g|\lesssim t^{-\frac12}$.
   
\end{proof}

Next, we introduce a local stability result for hyperbolic metrics using the above lemma. 

\begin{Lemma}[Local stability of hyperbolic metrics]\label{Lemma_local persist}
Let $(M,h_0)$ be a hyperbolic $3$-manifold of finite volume. Given any $k\in\mathbb{N}$, $D>0$, there exist $T_{loc}=T_{loc}(k,D)>0$ and $d_{loc}=d_{loc}(k, D)\leq D$ with the following property. Let $g(t)$ be a normalized Ricci-DeTurck flow defined on $M\times [0,T_{loc}]$ with initial metric $g(0)$. Suppose that  
    \begin{equation*}
        \Vert g(0)-h_0\Vert _{C^k(M)}\leq d_{loc}.
    \end{equation*}
    Then \begin{equation*}
        \Vert g(t)-h_0\Vert _{C^k(M)}< D\quad \forall t\in [0,T_{loc}].
    \end{equation*}
\end{Lemma}

\begin{proof}
    By Proposition 2.8 of \cite{Bamler}, for sufficiently small $d_{loc}$, there exist constants $T_{loc}, C_{loc}>0$, such that if $\Vert g(0)-h_0\Vert_{C^0(M)}\leq d_{loc}$, then a smooth solution $g(t)$ to the normalized Ricci-DeTurck flow exists on $[0,T_{loc}]$, and $$\Vert g(t)-h_0\Vert_{C^0(M)}\leq C_{loc}\Vert g(0)-h_0\Vert_{C^0(M)}\leq C_{loc}d_{loc}\quad \forall t\in [0,T_{loc}].$$
    
    When $k\geq 1$, since $g(0)-h_0$ is $C^k$,  we may assume $d_{loc}\leq \epsilon$ in Lemma~\ref{lemma_derivative_estimates} and obtain that
 \begin{equation*}
        \Vert g(t)-h_0\Vert_{C^k(M)}\leq C,\quad \Vert \nabla^{k+1}_{h_0}g(t)\Vert_{C^0(M)}\leq Ct^{-\frac12}.
    \end{equation*}
    Using the formula of $\frac{\partial}{\partial t}\nabla_{h_0}^kg(t)$ in \eqref{equ_derivative_lemma_4}, and applying the cut-off function and the maximum principle as in the previous lemma, we have \begin{equation*}
        \frac{\partial}{\partial t}|\nabla_{h_0}^k g(t)|\lesssim t^{-\frac12}+1.
    \end{equation*}
    Therefore, by integrating over $[0,T_{loc}]$, 
    \begin{equation*}
        \Vert g(t)-h_0\Vert_{C^k(M)}\leq  \Vert g(0)-h_0\Vert_{C^k(M)}+C'T_{loc}^{\frac12}+C'T_{loc}\leq d_{loc}+C'T_{loc}^{\frac12}+C'T_{loc}.
    \end{equation*}
    After possibly replacing $d_{loc}$ and $T_{loc}$ by smaller constants, we derive the desired result. 

\end{proof}

Next, we establish the following global stability theorem under $C^k$ perturbations. The proof is derived from the local stability and the global stability under $C^0$ perturbations of $h_0$ by Bamler \cite{Bamler}.

\begin{Th}[Stability of hyperbolic metrics under $C^k$ perturbations]\label{thm_stability_C^k}
Let $(M,h_0)$ be a hyperbolic $3$-manifold of finite volume. There is a constant $d_0$, such that if a metric $g(0)$ satisfies $\Vert g(0)-h_0\Vert _{C^0(M)}\leq d_0$, then the normalized Ricci-DeTurck flow $g(t)$ starting from $g(0)$ exists for all time. 

Furthermore, given an integer $k\geq 2$. 
For any $D>0$, there exists $d=d(k, D)\leq d_0$ with the following property. 
Let $g(t)$ be a normalized Ricci-DeTurck flow defined on $M\times [0, \infty)$ satisfying 
    \begin{equation*}
        \Vert g(0)-h_0\Vert _{C^k(M)}\leq d.
    \end{equation*}
Then \begin{equation*}
        \Vert g(t)-h_0\Vert _{C^k(M)}< D\quad \forall t\in [0,\infty).
    \end{equation*}
\end{Th}

\begin{proof} 
   Suppose by contradiction that there exist a sequence of normalized Ricci flows $g_n(t)$ defined on $M\times [0,\infty)$, and a sequence $d_n\rightarrow 0$ as $n\rightarrow\infty$, such that 
         \begin{equation*}
            \Vert g_n(0)-h_0\Vert _{C^k(M)}\leq d_n.
        \end{equation*}
    Moreover, there exists $t_n\in [0,\infty)$ such that  
    \begin{equation}\label{pf_persist_contrad}
        \Vert g_n(t_n)-h_0\Vert _{C^k(M)}\geq D.
    \end{equation}
    We also assume that $t_n$ is the minimum time for this property. 

   Let $T_{loc}$ and $d_{loc}$ be the constants provided by the local stability lemma. We also have $d_n\leq d_{loc}$ for sufficiently large $n$, which confirms the condition of Lemma~\ref{Lemma_local persist}. Thus, \begin{equation*}
        \Vert g_n(t)-h_0\Vert _{C^k(M)}< D \quad \forall t\in [0,T_{loc}].
    \end{equation*}
    This implies that $t_n>T_{loc}$.

    According to Section 6.2 of \cite{Bamler}, there exists a constant $C>0$ such that, if $\Vert g_n(0)-h_0\Vert _{C^0(M)}$ is sufficiently small (which can be ensured by choosing $n$ large enough), then
    \begin{equation*}
        \Vert g_n(t)-h_0\Vert _{C^0(M)}\leq C\Vert g_n(0)-h_0\Vert _{C^0(M)}\quad  \forall t\in [0,\infty).
    \end{equation*}
Furthermore, Corollary 2.7 of \cite{Bamler} provides the following estimate for order $m\in\mathbb{N}$. 
\begin{equation}\label{equ_m_derivative}
    \Vert \nabla_{h_0}^m g_n(t)\Vert _{C^{0}(M)}\leq C_mt^{-\frac{m}{2}}\Vert g_n(0)-h_0\Vert _{C^{0}(M)}  \quad \forall t\in (0,1],
\end{equation}
where $C_m>0$ is a constant independent of $n$. 
For $t>1$, it implies \begin{equation}\label{equ_m_derivative_2}
    \Vert \nabla_{h_0}^m g_n(t)\Vert _{C^{0}(M)}\leq C_m\Vert g_n(t-1)-h_0\Vert _{C^{0}(M)}\leq C_mC\Vert g_n(0)-h_0\Vert _{C^0(M)}.
\end{equation}

Since $\Vert g_n(0)-h_0\Vert _{C^{0}(M)}$ can be arbitrarily small for sufficiently large $n$, and since $t_n>T_{loc}$ stays away from zero, it follows from \eqref{equ_m_derivative} and \eqref{equ_m_derivative_2} that we can choose $n$ large enough so that for any $t\geq t_n$, \begin{equation*}
    \left\Vert g_n(t)-h_0\right\Vert _{C^k(M)}\leq D.
\end{equation*}
This contradicts the assumption \eqref{pf_persist_contrad}.
\end{proof}

\section{Long time behavior of Ricci-DeTurck flow}\label{section_rf_conv}
In this section, we review the long time behavior of the normalized Ricci-DeTurck flow and  its convergence toward the hyperbolic metric. In particular, we present a quantitative exponential decay estimate, which plays an essential role in the proof of Theorem~\ref{thm_volume_comparison}.
These results were originally introduced in \cite{Jiang-VargasPallete_RF}.

\subsection{Weighted little H{\"o}lder spaces}
First, we introduce weighted little H{\"o}lder spaces, and apply the interpolation theory. For closed hyperbolic $3$-manifolds, Knopf-Young \cite{Knopf-Young} studied the stability of the hyperbolic metric $h_0$ using Simonett's interpolation results \cite{Simonett}. They showed that starting from a metric in a little H{\"o}lder $\Vert \cdot\Vert _{2\alpha+\rho}$ neighborhood of $h_0$, the normalized Ricci-DeTurck flow converges exponentially fast in the $\Vert \cdot\Vert _{2+\rho}$ norm to $h_0$, where $\rho\in (0,1)$ and $\alpha\in (\frac{1}{2},1)$.

However, as explained in Section 5 of \cite{Jiang-VargasPallete_RF}, for the cusped manifolds, it is necessary to introduce an additional exponential weight in the thin part of the cusps.

To start our discussion, let $s>0$. For each $x\in M$, let $\Tilde{B}(x)\subset \mathbb{H}^3$ be the unit ball centered at a lift of $x$. For each tensor $l$ on $M$, the lift of $l$ on $\mathbb{H}^3$ is still denoted by $l$. We define the following weighted little H{\"o}lder spaces on $M$.

\begin{Def}[Weighted little H{\"o}lder spaces]\label{def_little_holder}
 Given $\lambda\in (0,1]$ and $s\geq 0$. 
    The \emph{weighted H{\"o}lder norm} $\Vert \cdot\Vert _{\mathfrak{h}^{k+\alpha}_{\lambda,s}}$ is defined as

\begin{align*}\label{equ_weighted norm}
    \Vert l\Vert _{\mathfrak{h}^{k+\alpha}_{\lambda,s}}: &= \sup_{x\in M} \ww_\lambda(x)\Vert l|_{\Tilde{B}(x)} \Vert_{\mathfrak{h}^{k+\alpha}}\\\nonumber
    &=\sup_{x\in M, 0\leq j\leq k} \left( \ww_\lambda(x)|\nabla^j\,l(x)| + \sup_{y_1\neq y_2\in \Tilde{B}(x)} \ww_\lambda(x)\frac{|\nabla^kl(y_1)-\nabla^kl(y_2)|}{d_{\Tilde{B}(x)}(y_1,y_2)^\alpha} \right)
\end{align*}
where 
\begin{align*}
\ww_\lambda(x) &=\begin{cases}
    e^{-\lambda r(x)}\quad \lambda\in (0,1),\\
    (r(x)+1)e^{-r(x)}\quad \lambda=1.
\end{cases}
\end{align*}
and
\begin{align*}
r(x) &=\begin{cases}
    0\quad\text{ if }x\in M(s),\\
    \text{dist}(x, \partial M(s))=\min_k (\text{dist}(x,T_k\times \{s\})\quad\text{otherwise}.
\end{cases}\\
\end{align*}
The $(r+1)$ multiplicative factor for $\ww_1$ is so that
\[\Vert l \Vert_{L^2(M)} \leq C_{\lambda,s} \Vert l\Vert _{\mathfrak{h}^{k+\alpha}_{\lambda,s}},
\]
holds.

As for fixed $\lambda$ the function $\ww_\lambda(x)$ satisfies
\[|\nabla^j \ww_\lambda(x)| \leq C_j \ww_\lambda(x)
\]
we can easily check that the norm $\Vert l \Vert_{\mathfrak{h}^{k+\alpha}_{\lambda,s}}$ is equivalent to
\[\sup_{x\in M, 0\leq j\leq k} \left( |\nabla^j(\ww_\lambda(x)\,l(x))| + \sup_{y_1\neq y_2\in \Tilde{B}(x)} \ww_\lambda(x)\frac{|\nabla^kl(y_1)-\nabla^kl(y_2)|}{d_{\Tilde{B}(x)}(y_1,y_2)^\alpha} \right)
\]

The \emph{little H{\"o}lder space} $\mathfrak{h}^{k+\alpha}_{\lambda,s}$ is defined to be the closure of $C_c^\infty$ symmetric covariant 2-tensors compactly supported in $M$ with respect to the weighted H{\"o}lder norm $\Vert \cdot\Vert _{\mathfrak{h}^{k+\alpha}_{\lambda,s}}$. 
\end{Def}

Moreover, for fixed $0<\sigma<\rho<1$, we define \begin{equation*}
    \mathcal{X}_{0}=\mathcal{X}_{0}(M,\rho,\lambda,s)=:\mathfrak{h}^{0+\rho}_{\lambda,s},\quad \mathcal{X}_{1}=\mathcal{X}_{1}(M,\rho,\lambda,s)=:\mathfrak{h}^{2+\rho}_{\lambda,s}.
\end{equation*} 

Next, we review the definition of interpolation spaces between $\mathcal{X}_0$ and $\mathcal{X}_1$. For further details, see \cite{Lunardi} and \cite{Triebel}.
For every $l\in \mathcal{X}_0 +\mathcal{X}_1$ and $t>0$, set \begin{equation*}
    K(t,l)=K(t,l;\mathcal{X}_0,\mathcal{X}_1):=\inf_{l=l_0+l_1,\, l_i\in\mathcal{X}_i}\left(\Vert l_0\Vert _{\mathcal{X}_0}+t\Vert l_1\Vert _{\mathcal{X}_1}\right).
\end{equation*}
For each $t$, it defines an equivalent norm for the space $\mathcal{X}_0+\mathcal{X}_1$. 
\begin{Def}[Interpolation spaces]\label{def_real/cts_interpolation}
    Let $0<\theta<1$, $1\leq p\leq \infty$, and define the following \emph{real interpolation spaces} between $\mathcal{X}_0$ and $\mathcal{X}_1$: 
        \item \begin{equation*}
            (\mathcal{X}_0,\mathcal{X}_1)_{\theta,p}:=\left\{l\in \mathcal{X}_0+\mathcal{X}_1: t\mapsto t^{-\theta}K(t,l)\in L^p_*(0,\infty)\right\},
        \end{equation*} 
        where $L^p_*$ is the $L^p$ space with respect to the measure $dt/t$. Note that the $L^\infty_*$ space coincides with the standard $L^\infty$ space. 
        The norm of $l\in (\mathcal{X}_0,\mathcal{X}_1)_{\theta,p}$ is given by \begin{equation*}
        \Vert l\Vert _{(\mathcal{X}_0,\mathcal{X}_1)_{\theta,p}}:= \Vert t^{-\theta}K(t,l)\Vert _{L^p_*(0,\infty)}.
        \end{equation*}
        Moreover, the \emph{continuous interpolation space} between $\mathcal{X}_0$ and $\mathcal{X}_1$ is defined as follows. \begin{equation*}
            (\mathcal{X}_0,\mathcal{X}_1)_{\theta}:= \left\{l\in \mathcal{X}_0+\mathcal{X}_1: \lim_{t\rightarrow 0^+}t^{-\theta}K(t,l)=\lim_{t\rightarrow\infty}t^{-\theta}K(t,l)=0\right\}.
        \end{equation*}
\end{Def}
Observe that the function $K(t,x)$ is continuous in terms of $t$, thus $(\mathcal{X}_0,\mathcal{X}_1)_{\theta}$ is a closed subspace of $(\mathcal{X}_0,\mathcal{X}_1)_{\theta,\infty}$ and it is endowed with the $(\mathcal{X}_0,\mathcal{X}_1)_{\theta,\infty}$-norm.

Let $\alpha\in (0,1)$ with $2\alpha+\rho\notin\mathbb{N}$, 
consider the continuous interpolation space $\mathcal{X}_{\alpha}:=(\mathcal{X}_{0}, \mathcal{X}_{1})_\alpha$, \cite[Corollary 5.3]{Jiang-VargasPallete_RF} proves that \begin{equation*}\label{equ_X_alpha}
    \mathcal{X}_{\alpha}=(\mathcal{X}_{0}, \mathcal{X}_{1})_\alpha \cong \mathfrak{h}_{\lambda,s}^{2\alpha +\rho}.
\end{equation*}

\subsection{Exponential attractivity}
We prove the following the exponential attractivity toward the hyperbolic metric, which uses the method of \cite[Theorem 1.1]{Jiang-VargasPallete_RF}.
It will be the key tool for proving the main theorem.
\begin{Th}\label{thm_ricci_flow}
Let $(M,h_0)$ be a hyperbolic $3$-manifold of finite volume, and let $\alpha\in (0,\frac{1-\rho}{2})\cup(\frac{1-\rho}{2},\frac12)$. Given $\lambda\in (0,1]$. For every
    $\omega \in (0,\lambda(2-\lambda))$, 
    there exist $\rho_0,c>0$, such that if $h$ is a metric on $M$ with 
    \begin{equation*}
  \Vert h-h_0\Vert _{C^2(M)}<\rho_0,
\end{equation*} then the solution $h(t)$ of the normalized Ricci-DeTurck flow starting at $h(0)=h$ exists for all time. 
Moreover, we have \begin{equation*}
    \Vert h(t)-h_0\Vert _{\mathcal{X}_{1}}\leq \frac{c}{t^{1-\alpha}} e^{-\omega t}\Vert h-h_0\Vert _{C^2(M)}, \quad \forall t>0.
\end{equation*}
\end{Th}

\begin{proof}
We first review some notations from \cite{Jiang-VargasPallete_RF}.
Define \begin{align}\label{equ_C_alpha}
    C_\alpha^0\left((0,\infty),\mathcal{X}_0\right):=&\left\{F\in C^0((0,\infty),\mathcal{X}_0): \lim_{t\rightarrow 0}t^{1-\alpha}\Vert F(t)\Vert _{\mathcal{X}_0}=0\right\},\\\nonumber
    C_\alpha^1\left((0,\infty),\mathcal{X}_0,\mathcal{X}_1\right):=&\Big\{g\in C^1\left((0,\infty),\mathcal{X}_0\right)\cap C^0\left((0,\infty),\mathcal{X}_1\right):\\\nonumber
    &\lim_{t\rightarrow 0}t^{1-\alpha}\left(\Vert g'(t)\Vert _{\mathcal{X}_0}+\Vert g(t)\Vert_{\mathcal{X}_1}\right)=0\Big\}.
\end{align}
Consider the linear problem \begin{equation}\label{equ_linear_A}
    \frac{\partial}{\partial t}g(t)=A g(t)+F(t),
\end{equation}
with initial data $g(0)$. The map $g(t)\mapsto g(0)$ is denoted by $I_\alpha$. 
Let \begin{align*}
    \mathcal{H}(\mathcal{X}_1,\mathcal{X}_0):=&\big\{A\in \mathcal{L}(\mathcal{X}_1,\mathcal{X}_0):\\
    &A\text{ generates a strongly continuous analytic semigroup}\big\},\\
    \mathcal{M}_\alpha(\mathcal{X}_1,\mathcal{X}_0):=& \Big\{A\in \mathcal{H}(\mathcal{X}_1,\mathcal{X}_0): \\
    &(\partial_t-A,I_\alpha)\in Isom \left(C_\alpha^1((0,\infty),\mathcal{X}_0,\mathcal{X}_1),C_\alpha^0((0,\infty),\mathcal{X}_0)\times \mathcal{X}_\alpha\right)\Big\}.
\end{align*}
In other words, $\mathcal{M}_\alpha(\mathcal{X}_1,\mathcal{X}_0)\subset \mathcal{H}(\mathcal{X}_1,\mathcal{X}_0)$ consists of the operators for which the differential equation \eqref{equ_linear_A} admits a unique solution $g(t)\in C_\alpha^1\left((0,\infty),\mathcal{X}_0,\mathcal{X}_1\right)$ for any given pair $(F,g(0))\in C_\alpha^0\left((0,\infty),\mathcal{X}_0\right)\times \mathcal{X}_\alpha$. 

Suppose that $A_{h_0}\in \mathcal{M}_\alpha(\mathcal{X}_1,\mathcal{X}_0)$,
the stability theorem for the Ricci-DeTurck flow then allows us to express the solution as \begin{equation}\label{equ_integral_sol}
    h(t)= e^{tA_{h_0}} h(0)+\int_0^t e^{(t-s)A_{h_0}}(\mathcal{A}(h(s))-A_{h_0})h(s)\,ds,
\end{equation}
where $\mathcal{A}(h)$ will denote the right-hand side of the normalized Ricci-DeTurck flow, $A_{h_0}$ (which will take the role of $A$ in \eqref{equ_linear_A}) the linearization of the flow at the fixed hyperbolic metric and $(\mathcal{A}(h(t))-A_{h_0})h(t)$ will take the role of $F(t)$ in \eqref{equ_linear_A}. 

As argued in \cite{Jiang-VargasPallete_RF}, we need to verify $A_{h_0}\in\mathcal{M}_\alpha(\mathcal{X}_1,\mathcal{X}_0)$, and then use the form \eqref{equ_integral_sol} to derive the exponential attractivity.

Let $\epsilon>0$ be a sufficiently small constant. Applying the stability theorem (Theorem~\ref{thm_stability_C^k}) with order $k=2$, we obtain a constant $\delta>0$ such that if the initial tensor $h(0)\in \mathcal{X}_\alpha$ is in the $\delta$-neighborhood of $h_0$ in $C^{2}$, then the corresponding normalized Ricci-DeTurck flow $h(t)$ remains in the $\epsilon$-neighborhood of $h_0$ in $C^2$ for all time.
In particular we have that $A_{h(t)}\in \mathcal{L}(\mathcal{X}_1,\mathcal{X}_0)$.

By Lemma~\ref{lemma_derivative_estimates},  \begin{equation*}
    \Vert \nabla^3h(t)\Vert _{C^0(M)}\lesssim t^{-\frac12},\quad\forall t\in (0,1].
\end{equation*}
Thus we have $h(t)\in C^1\left((0,\infty),\mathcal{X}_0\right)\cap C^0\left((0,\infty),\mathcal{X}_1\right)$, and since $\alpha<\frac12$, \begin{equation*}
    \lim_{t\rightarrow 0}t^{1-\alpha}\left(\Vert h'(t)\Vert _{\mathcal{X}_0}+\Vert h(t)\Vert _{\mathcal{X}_1}\right)\lesssim \lim_{t\rightarrow 0}t^{1-\alpha}\Vert h(t)\Vert _{C^3(M)}\lesssim \lim_{t\rightarrow 0}t^{1-\alpha}t^{-\frac12}=0.
\end{equation*}
It shows that $h(t)\in C_\alpha^1\left((0,\infty),\mathcal{X}_0,\mathcal{X}_1\right)$, as defined in \eqref{equ_C_alpha}.
Moreover, one can easily see that $F(t):=\left(\mathcal{A}(h(t))-A_{h_0}\right)h(t)\in C^0\left((0,\infty),\mathcal{X}_0\right)$. Moreover, \begin{equation*}
    \lim_{t\rightarrow 0}t^{1-\alpha}\Vert F(t)\Vert _{\mathcal{X}_0}\lesssim  \lim_{t\rightarrow 0}t^{1-\alpha}\Vert h(t)\Vert _{\mathcal{X}_1}\lesssim  \lim_{t\rightarrow 0}t^{1-\alpha}\Vert h(t)\Vert _{C^3(M)}\lesssim \lim_{t\rightarrow 0}t^{1-\alpha}t^{-\frac12}=0,
\end{equation*}
which implies $F(t)\in C_\alpha^0\left((0,\infty),\mathcal{X}_0\right)$.

\cite[Section 6]{Jiang-VargasPallete_RF} shows that $A_{h_0}\in\mathcal{H}(\mathcal{X}_1,\mathcal{X}_0)$. Combined with the argument above, this implies that $A_{h_0}\in \mathcal{M}_\alpha(\mathcal{X}_1,\mathcal{X}_0)$. 
As a consequence, the maximal regularity property implies that there exists solution $H(t)\in C_\alpha^1\left((0,\infty),\mathcal{X}_0,\mathcal{X}_1\right)$ to the linear equation 
\begin{equation*}
\begin{cases}
    \dfrac{\partial}{\partial t}H(t)=A_{h_0}H(t)+\left(\mathcal{A}(h(t))-A_{h_0}\right)H(t),\\
    H(0)=h(0).
\end{cases}
\end{equation*}
Such solution can be expressed by the integral formula \begin{equation*}
    H(t):=e^{tA_{h_0}} h(0)+\int_0^t e^{(t-s)A_{h_0}}\left(\mathcal{A}(h(s))-A_{h_0}\right)h(s)\,ds,
\end{equation*}
for $t\in [0,\infty)$.

We observe that $h(t)\in C_\alpha^1\left((0,\infty),\mathcal{X}_0,\mathcal{X}_1\right)$ also solves the linear system, and hence $h(t)=H(t)$ for all $t\in [0,\infty)$. In other words, the DeTurck flow $h(t)$ takes the following form.
\begin{equation*}
    h(t)=e^{tA_{h_0}} h(0)+\int_0^t e^{(t-s)A_{h_0}}\left(\mathcal{A}(h(s))-A_{h_0}\right)h(s)\,ds.
\end{equation*}
Let $l(t):=h(t)-h_0$. We obtain 
\begin{equation}\label{equ_l_A_exp}
    l(t)=e^{tA_{h_0}} l(0)+\int_0^t e^{(t-s)A_{h_0}}\left(\mathcal{A}(h(s))(h(s))-\mathcal{A}(h_0)(h_0)-A_{h_0}(l(s))
    \right)\,ds.
\end{equation}

Using the same estimates of \eqref{equ_l_A_exp} in \cite[Section 7]{Jiang-VargasPallete_RF}, which relies on the fact that any complex number $\omega_c$ with $\Re(\omega_c) > -\lambda(2-\lambda)$ lies in the resolvent set of $A_{h_0}$, we obtain the desired result. Specifically, for every real number $\omega\in (0,\lambda(2-\lambda))$, there exists constants $\rho, c, c'>0$, such that if $\Vert l(0)\Vert_{C^2(M)}=\Vert h-h_0\Vert_{C^2(M)}<\rho_0$, then
\begin{equation*}
    \Vert l(t)\Vert _{\mathcal{X}_1}\leq \frac{c'}{t^{1-\alpha}} e^{-\omega t}\Vert l(0)\Vert _{\mathcal{X}_\alpha}\leq \frac{c}{t^{1-\alpha}}e^{-\omega t}\Vert l(0)\Vert _{C^2(M)},\quad \forall t> 0.
\end{equation*}
\end{proof}

In Section~\ref{section_proof}, given any metric on $M$ with scalar curvature bounded below by $-6$, we will construct a new metric that is sufficiently close to the hyperbolic metric while carefully tracking the change in volume. 
This allows us to invoke the above theorem.

\section{volume comparison}\label{section_proof}

In this section, we present the proof of Theorem~\ref{thm_volume_comparison}. 
First, given an arbitrary metric $h$, we check the condition of Theorem~\ref{thm_ricci_flow} by applying the Ricci flow with bubbling-off. Our goal is to find a finite time at which the evolution of the metric, starting from $h$, becomes sufficiently close to $h_0$ in $C^2$. However, achieving this is not always possible. First, for a general initial metric $h$, there may be neither long-time nor short-time existence of the Ricci flow. Therefore, to construct the Ricci flow in such settings, we approximate $h$ by a sequence of cusp-like metrics $\{h_i\}$, and run the normalized Ricci flow starting from each $h_i$.
Furthermore, if $h$ is asymptotic to a hyperbolic metric in the cusp that differs from $h_0$, then according to the stability of cusp-like structures (Theorem~\ref{thm_stability_cusp}), this asymptotic behavior persists for all time. Consequently, $h(t)$ remains distant from $h_0$ and never becomes close in $C^2$. 
To address this issue, we will define each new metric $h_i$ as asymptotically to $h_0$ at the cuspidal end.

 \subsection{Mixed flows and exponential attractivity}
We start by considering two special cases:
\begin{enumerate}[(I)]
    \item Let $\epsilon>0$ be sufficiently small. According to Theorem~\ref{thm_ricci_flow}, 
 if $h$ satisfies $\Vert h-h_0\Vert _{C^2(M)}\leq\epsilon$, then the long-time existence of the normalized Ricci-DeTurck flow was established in that theorem.
 
    \item In a different setting, if $h$ is asymptotically cusped of order $k\geq 2$, then by Theorem~\ref{thm_stability_cusp_asymp}, there exists a normalized Ricci flow with bubbling-off on $M$ starting from $h$, defined for all time.

\end{enumerate}
We will examine these two cases in greater detail in the rigidity part of the proposition in Section~\ref{subsection_rigidity}.

For the general case, we choose a sequence $\{s_i\}$ with $s_i\rightarrow\infty$ as $i\rightarrow\infty$. 
Then, we define a new metric $h_i$ on $M$ using $s_i$, such that 
\begin{align}\label{equ_h_i}
    \bullet & \,h_i=h\text{ on the thick part }M(s_i),\\\nonumber
    \bullet & \,h_i=h_0\text{ on the thin part }M\setminus M(2s_i)=\cup_j T_j\times (2s_i,\infty),\\\nonumber
    \bullet &\, h_i\text{ is a smooth interpolation between the metrics }h\text{ and }h_0\text{ on }M(2s_i)\setminus M(s_i)\\\nonumber
    &=\cup_j T_j\times (s_i,2s_i], \text{ and }R(h_i)\geq -6.
\end{align}

The volume of $M$ satisfies that \begin{equation*}
    \vol_h(M)=\lim_{i\rightarrow\infty} \vol_{h}(M(s_i))=\lim_{i\rightarrow\infty} \vol_{h_i}(M(s_i)). 
\end{equation*}

For each $i\in\mathbb{N}$, suppose that $h_i(t)$ solves the normalized Ricci flow \eqref{RF}, starting with $h_i(0)=h_i$.

Recall the notion of Ricci flow with bubbling-off in Section~\ref{subsection_bubbling}. Since our initial metric $h_i$ is identical to $h_0$ on the thin part $M \setminus M(2s_i)$, 
it possesses a cusp-like structure, which permits us to perform Ricci flow with bubbling-off on $M$ starting at $h_i$. According to Theorem~\ref{thm_stability_cusp}, $h_i(t)$ exists for all time and remains asymptotic to $h_0$ at infinity in the cuspidal end in $C^k$, uniformly for all time $t\in [0,\infty)$.

Furthermore, because of the reduction in volume through surgery, there can only be a finite number of surgeries \cite[Section 3]{Bessieres-Besson-Maillot}. The only possible surgeries are pinching off inessential $\delta$-necks and attaching $\delta$-almost standard caps. This finite number is represented as $m_i\in \mathbb{N}$, and the last singular time is denoted by $t_i^{m_i}$.

For each $i\in\mathbb{N}$, let $\mathcal{X}_j$, $j=0,1$, be the weighted H{\"o}lder spaces where the weight is applied starting at $s_i$, we obtain the following corollary from Theorem~\ref{thm_ricci_flow}.
\begin{Cor}\label{cor_simonett}
      Given $\lambda\in (0,1]$. For every
    $\omega \in (0,\lambda(2-\lambda))$, there exist $\rho_i, c>0$, such that if $g_i$ is a metric on $M$ with \begin{equation*}
    \Vert g_i-h_0\Vert _{C^2(M)}<\rho_i,
\end{equation*} then the solution $g_i(t)$ of the normalized Ricci-DeTurck flow starting at $g_i(0)=g_i$ satisfies \begin{equation*}
    \Vert g_i(t)-h_0\Vert _{\mathfrak{h}^{2+\rho}_{s_i}(M)}\leq \frac{c\rho_i}{t^{1-\alpha}} e^{-\omega t},\quad \forall t>0.
\end{equation*} 
\end{Cor}
Note that by choosing $\rho_i$ sufficiently small, we can ensure that the constant $c$ in the corollary does not depend on $i$.

Due to the convergence of $h_i(t)$ toward $h_0$ on the thick part (\cite[Theorem 1.2]{Bessieres-Besson-Maillot}), there exists a post-surgery time $t_i>t_i^{m_i}$ 
such that \begin{equation*}
   \Vert h_i(t_i)-h_0\Vert _{C^2(M(s_i))}<\rho_i.
\end{equation*}
If, on the thin part $M\setminus M(s_i)$, $h_i(t_i)$ is not in the $C^2$-neighborhood of $h_0$ of radius $\rho_i$, we replace $h_i(t_i)$ with $h_{i+}(t_i)$ on $M\setminus M(s_i)$ so that the new metric agrees with $h_0$ on a further thin part, and it satisfies \begin{equation*}
    \Vert h_{i+}(t_i)-h_0\Vert _{C^2(M)}<\rho_i.
\end{equation*}
This verifies the condition of Corollary~\ref{cor_simonett}. Therefore, we get
\begin{equation*}
    \Vert h_i(t)-h_0\Vert _{\mathfrak{h}^{2+\rho}_{s_i}(M)}\leq \frac{c\rho_i}{(t-t_i)^{1-\alpha}}e^{-\omega (t-t_i)},\quad \forall t> t_i,
\end{equation*}

Now we redefine $h_i(t)$ as a mixed flow:
For $0\leq t< t_i$, $h_i(t)$ is still the normalized Ricci flow. And for $t\geq t_i$, it solves the normalized Ricci-DeTurck flow starting with $h_i(t_i):=h_{i+}(t_i)$.

\subsection{Volume comparison}
We now compare the volume of each thick region $M(s_i)$ with respect to different metrics using the mixed flow $h_i(t)$. 
The volume function $\vol_{h_i(t)}(M(s_i))$ is differentiable almost everywhere on both intervals $[0,t_i)$ and $[t_i,\infty)$. If $t$ is non-singular time, 
then we have \begin{align*}
    &\frac{d}{dt}\vol_{h_i(t)}(M(s_i))\\
    =&\int_{M(s_i)}\frac{d}{ds}\Big|_{s=t}\sqrt{\text{det}_{h_i(t)}h_i(s)}\dvol\\\nonumber
    =&\frac{1}{2}\int_{M(s_i)}\tr_{h_i(t)}\left(\frac{d}{ds}\Big|_{s=t}\,h_i(s)\right)\dvol\\\nonumber
    =& 
    \begin{cases}
        \displaystyle-\int_{M(s_i)} \big(R(h_i(t))+6\big)\dvol &t< t_i,\\
       \displaystyle -\int_{M(s_i)} \big(R(h_i(t))+6\big)\dvol+\int_{\partial M(s_i)}\langle V(h_i(\tau)),\nu\rangle \dvol &t\geq t_i.
    \end{cases}
\end{align*}
Therefore, the volume of $M(s_i)$ with respect to the metric $h_i(t)$ satisfies the following inequality. For $t\geq t_i$,
 \begin{align*}
    \vol_{h_i(t)}(M(s_i))=&\vol_{h_i(t_i)}(M(s_i))-\int_{t_i}^t\int_{M(s_i)} (R(h_i(\tau))+6)\dvol\,d\tau\\
    &+\int_{t_i}^t\int_{\partial M(s_i)}\langle V(h_i(\tau)),\nu\rangle \dvol\,d\tau.
\end{align*}
Note that $ \vol_{h_i(t_i)}(M(s_i))=\vol_{h_{i+}(t_i)}(M(s_i))$, and the inequality $R(h_i(t))\geq -6$ is preserved by the normalized Ricci flow and DeTurck flow.
Moreover, since surgeries can only decrease volume, and it also preserves $R(h_i(t))\geq -6$ (Definition 4.4.3, \cite{BBB+10}), for $t\geq t_i$ we have 
 \begin{align}\label{equ_vol_M(s)}
   & \vol_{h_i(t)}(M(s_i))\\\nonumber
    \leq& \vol_{h_i}(M(s_i))-\int_{0}^t\int_{M(s_i)} (R(h_i(\tau))+6)\dvol\,d\tau+\int_{t_i}^t\int_{\partial M(s_i)}\langle V(h_i(\tau)),\nu\rangle \dvol\,d\tau \\\nonumber 
    \leq&\vol_{h}(M(s_i))+\int_{t_i}^t\int_{\partial M(s_i)}\langle V(h_i(\tau)),\nu\rangle \dvol\,d\tau.
\end{align}

We estimate $ |V(h_i(\tau))|$ using the following lemma.
\begin{Lemma}\label{lemma_deturck}
Fix $\alpha\in (0,\frac{1-\rho}{2})\cup(\frac{1-\rho}{2},\frac12)$, and choose $\lambda\in (0,1)$.
    Let $h(t)$ be a normalized Ricci-DeTurck flow satisfying the assumptions of Theorem~\ref{thm_ricci_flow}, where the little H{\"o}lder spaces are defined with spatial parameter $s>0$. Then for any $\omega<\min(1-\lambda^2,\lambda(2-\lambda))$, there exists a constant $C=C(\lambda,\omega,\alpha)>0$ so that
    \begin{equation*}
        |V(h(x,t))|\leq Ce^{-\omega t+\lambda r(x)}\quad \forall x\in M,\,t\geq 0.
    \end{equation*}
\end{Lemma}

\begin{proof}
Let $l(t)=h(t)-h_0$.
    According to the calculation in Lemma 3.2 of \cite{Hu-Ji-Shi}, \begin{equation}\label{lemma4.2_1}
        \frac{\partial}{\partial t} V_j = \Delta_h V_j+R_j^kV_k+\left(\frac{\partial}{\partial t}h(t)_{jk}h(t)^{pq}+h(t)_{jk} \frac{\partial}{\partial t}h(t)^{pq}\right)\left(\Gamma_{pq}^k-(\Gamma_{h_0})_{pq}^k\right),
    \end{equation}
    and
    \begin{align*}
        & \left|\Gamma_{pq}^k-(\Gamma_{h_0})_{pq}^k\right|(x)\leq K_1\left(\Vert \nabla_{h_0}h(t)\Vert _{C^0(M,h_0)}|l(t)|(x)+|\nabla_{h_0}h(t)|(x)\right),\\
        & \left|\frac{\partial}{\partial t}h(t)\right|(x)\leq K_2 \left(|\nabla_{h_0}^2h(t)|(x)+|\nabla_{h_0}h(t)|(x)+|l(t)|(x)\right).\\
    \end{align*}
    Setting $\lambda\in(0,1)$ and $\omega'\in(0,\lambda(2-\lambda))$, then Theorem~\ref{thm_ricci_flow} yields a constant $\rho_0=\rho_0(\lambda,\omega')$. We can assume $\rho_0\leq d$ in Theorem~\ref{thm_stability_C^k} and apply the theorem with order $2$.
    Since $\Vert h-h_0\Vert_{C^2(M)}<\rho_0$, for each $t\geq 0$, $h(t)$ stays close to $h_0$ in $C^2$, hence we can choose the constants $K_1$ and $K_2$ depending only on $h_0,\lambda$ and $\omega'$. We will omit the dependence on $h_0$ from now on. 
    This stability result also provides a constant $K_3=K_3(\lambda,\omega')$, such that \begin{equation}
        \left|\frac{\partial}{\partial t}h(t)_{jk}h(t)^{pq}+h(t)_{jk} \frac{\partial}{\partial t}h(t)^{pq}\right|\leq K_3.
    \end{equation}
    Applying Theorem~\ref{thm_ricci_flow}, we obtain \begin{equation}
         \left|\Gamma_{pq}^k-(\Gamma_{h_0})_{pq}^k\right|(x)\leq K_4\frac{c\rho_0}{t^{1-\alpha}}e^{-\omega' t}e^{\lambda r(x)}.
    \end{equation}  
    Moreover, we may assume that $\rho_0\leq \frac12$, from which we have
    \begin{equation}\label{lemma4.2_2}
        |R_j^k+2\delta_j^k|=|R_j^k-(R(h_0))_j^k|\lesssim \Vert l(0)\Vert _{C^2(M)}\leq \frac12,
    \end{equation}
    where  $K_4=K_4(\lambda,\omega')$.

    Combining \eqref{lemma4.2_1}-\eqref{lemma4.2_2}, we get 
    \begin{equation*}
        \frac{\partial}{\partial t}|V|\leq \Delta_h |V|-\frac32|V|+\frac{K_5}{t^{1-\alpha}}e^{-\omega' t}e^{\lambda r(x)},
    \end{equation*}
    where $K_5\geq K_3K_4c\rho_0$. Note that we need to avoid any dependence on $\rho_0$. In particular, when combining this lemma with Corollary~\ref{cor_simonett}, we do not want the constants to depend on $i$. We may set $K_5=\frac12 K_3K_4c$, and assume that the constant $\rho_i$ obtained from the corollary satisfies $\rho_i\leq \frac12$ for all sufficiently large $i$.

    Let $v=e^{-\lambda r(x)}|V|$, then it satisfies the following inequality. \begin{equation*}
        \frac{\partial}{\partial t} v \leq \Delta_h v+2\lambda\nabla_h r\cdot\nabla_h v-\left(-\lambda\Delta_h r+\frac32-\lambda^2|\nabla_h r|^2\right)v +\frac{K_5}{t^{1-\alpha}}e^{-\omega' t}.
    \end{equation*}
    If $x\in \text{Int}M(s)$, then $r(x)=0$ and $\Delta_h r=\nabla_h r=0$. Otherwise, since $h$ is close to $h_0$ in $C^2$, we have $\nabla_h r=1+O(\Vert l(0)\Vert _{C^2(M)})$ and $\Delta_h r=O(\Vert l(0)\Vert _{C^2(M)})$. Therefore, for all $x\in M$, we have the inequality below. 
    \begin{align*}
        \frac{\partial}{\partial t} v
        \leq & \Delta_h v+2\lambda\nabla_h r\cdot\nabla_h v-\left(\frac32-\lambda^2-O(\Vert l(0)\Vert _{C^2(M)})\right)v +\frac{K_5}{t^{1-\alpha}}e^{-\omega' t}\\
        \leq& \Delta_h v+2\lambda\nabla_h r\cdot\nabla_h v-\left(1-\lambda^2\right)v+\frac{K_5}{t^{1-\alpha}}e^{-\omega' t},
    \end{align*}
   where $1-\lambda^2>0$. Solve the ODE for $t\geq 0$
    $$
    \begin{cases}
        \dfrac{du}{dt}=-\left(1-\lambda^2\right)u+\dfrac{K_5}{t^{1-\alpha}}e^{-\omega' t}, \\
        u(0)=\Vert v(0)\Vert_{C^0(M)}.
    \end{cases}
    $$
    We obtain
    \begin{equation*}
        u(t)= \left(u(0)+K_5\int_{0}^t\frac{e^{(1-\lambda^2-\omega')\tau}}{\tau^{1-\alpha}}\,d\tau\right) e^{-(1-\lambda^2)t}.
    \end{equation*}
   Assume that $\omega'<1-\lambda^2$, we have \ \begin{align*}
        u(t)\leq&\left(u(0)+K_5e^{(1-\lambda^2-\omega') t}\int_{0}^t\frac{1}{\tau^{1-\alpha}}\,d\tau\right) e^{-(1-\lambda^2)t}\\
        = & \Vert v(0)\Vert_{C^0(M)}e^{-(1-\lambda^2)t}+\frac{K_5}{\alpha}t^\alpha e^{-\omega' t}.
    \end{align*}
    From \eqref{eq_V}, observe that $\Vert v(0)\Vert_{C^0(M)}\lesssim \rho_0<1$, and let $\omega<\omega'$, then there exists a constant $C=C(\lambda,\omega,\alpha)$, such that $$\left(1+\frac{K_5}{\alpha}t^\alpha\right) e^{-\omega' t}\leq Ce^{-\omega t}.$$
    This implies that \begin{equation*}
        u(t)\leq Ce^{-\omega t},\quad \forall \omega<\min\left(1-\lambda^2,\lambda(2-\lambda)\right).
    \end{equation*}
   
    Furthermore, according to the maximum principle (see for instance Lemma 4.2 in \cite{Qing-Shi-Wu}), we have $v(\cdot,t)\leq u(t)$. Therefore, the following holds for all $t\geq 0$:
    \begin{equation*}
        |V(h(x,t))|\leq Ce^{-\omega t+\lambda r(x)}.
    \end{equation*}
\end{proof}

Note that the constant $C$ in Lemma~\ref{lemma_deturck} is independent of $i$ when the lemma is combined with Corollary~\ref{cor_simonett}.
Substituting the result into \eqref{equ_vol_M(s)}, we have
\begin{align*}
    \vol_{h_i(t)}(M(s_i))-\vol_{h}(M(s_i))\leq& \int_{t_i}^t\int_{\partial M(s_i)} Ce^{-\omega (\tau-t_i)}\dvol\,d\tau\\
    \lesssim &\vol(\cup_jT_j\times \{s_i\})\int_{t_i}^te^{-\omega (\tau-t_i)}\,d\tau\\
    \lesssim& e^{-2s_i}(1-e^{-\omega (t-t_i)})\leq e^{-2s_i},
\end{align*}
for each $\omega<\min\left(1-\lambda^2,\lambda(2-\lambda)\right)$. 

Fix an arbitrary $\epsilon>0$. Since $s_i\to\infty$ as $i\to\infty$, there exists sufficiently large $i_0\in\mathbb{N}$ and $t_0>0$, such that for any $i\geq i_0$ and $t\geq t_0$, we have \begin{equation*}
     \vol_{h_i(t)}(M(s_i))-\vol_{h}(M(s_i))<\frac \epsilon 2,
\end{equation*}
and
\begin{align*}
\left|\vol_{h_i(t)}(M(s_i))-\vol_{h_0}(M(s_i))\right|
 \leq&   \int_{M(s_i)} \left|\sqrt{\text{det}_{h_0}(h_i(t))}-1\right|\dvol_{h_0}\\
 \approx&\frac12\int_{M(s_i)}\left|\tr_{h_0}\left(h_i(t)-h_0\right)\right|\dvol_{h_0}\\
 \lesssim& \int_{M(s_i)}e^{-\omega (t-t_i)}\dvol_{h_0}\\
 \lesssim& e^{-\omega (t-t_i)} \vol_{h_0}(M)<\frac \epsilon 2.
\end{align*}
Therefore, \begin{equation*}
    \vol_{h_0}(M(s_i))<\vol_{h}(M(s_i))+\epsilon.
\end{equation*}
This implies \begin{equation*}
    \vol_{h_0}(M)=\lim_{i\to\infty} \vol_{h_0}(M(s_i))\leq \lim_{i\to\infty}\vol_{h}(M(s_i))=\vol_h(M).
\end{equation*}

\subsection{Proof of rigidity}\label{subsection_rigidity}
\begin{enumerate}[(I)]
    \item 
Let $\epsilon$ be a sufficiently small constant. By Theorem~\ref{thm_ricci_flow},
 if $h$ satisfies $\Vert h-h_0\Vert _{C^2(M)}\leq \epsilon$, then the long-time existence of the normalized Ricci-DeTurck flow was established in that theorem. Denote the DeTurck flow starting from $h$ by $h(t)$. Moreover, we have $$\Vert h(t)-h_0\Vert_{\mathfrak{h}_{\lambda,s_0}^{2+\rho}}\leq \frac{c\epsilon}{t^{1-\alpha}}e^{-\omega t},$$
 where $s_0>0$ is a fixed spatial parameter. Applying Lemma~\ref{lemma_deturck}, we obtain a constant $C$ so that 
  \begin{equation*}
        |V(h(x,t))|\leq Ce^{-\omega t+\lambda r(x)}\quad \forall x\in M,\,t\geq 0.
    \end{equation*}
    Since $h$ is close to $h_0$ in $C^2$, for $x\in \cup_jT_j\times\{s\}$ with $s\geq s_0$, $r(x)$ is approximately $s-s_0$, and therefore bounded above by $2(s-s_0)$.
Hence, the volume of $M(s)$ with respect to $h(t)$ satisfies the following inequality. 
\begin{align*}
   & \vol_{h(t)}(M(s))-\vol_{h}(M(s))+\int_{0}^t\int_{M(s)} (R(h(\tau))+6)\dvol\,d\tau\\\nonumber
   =&\int_{0}^t\int_{\partial M(s)}\langle V(h(\tau)),\nu\rangle \dvol\,d\tau\\\nonumber
    \lesssim &e^{-2s+2\lambda(s-s_0)}\int_{0}^te^{-\omega \tau}\,d\tau\\
    \lesssim & e^{-2(1-\lambda)s}\to 0,\quad s\to\infty.
\end{align*}
Then as argued previously, \begin{equation*}
    \vol_{h_0}(M)= \vol_h(M)-\int_{0}^\infty\int_{M} (R(h(\tau))+6)\dvol\,d\tau\leq \vol_h(M).
\end{equation*}

Suppose that the equality $\vol_{h_0}(M)=\vol_h(M)$ holds, it implies that \begin{equation*}
    \int_{M} (R(h(t))+6)\dvol =0\quad\forall t\geq 0.
\end{equation*}
Since $R(h(t))\geq -6$, we obtain that \begin{equation*}
   R(h(t))\equiv -6.
\end{equation*}
Consider the corresponding normalized Ricci flow $\Tilde{h}(t)=\Phi(t)^*h(t)$ with $\Tilde{h}(0)=h$, where $\Phi(t)$ is a family of diffeomorphisms on $M$ with $\Phi(0)=Id$. 
Under the Ricci flow, we also have \begin{equation*}
     R(\Tilde{h}(t))\equiv -6.
\end{equation*}
Together with the evolution equation of the scalar curvature \begin{equation*}
    \frac{d}{dt}R(\Tilde{h}(t))=\Delta R(\Tilde{h}(t))+2|Ric(\Tilde{h}(t))|^2+4R(\Tilde{h}(t)),
\end{equation*}
it shows that $Ric(h)\equiv -2h$ on $M$. Consequently, $h$ is hyperbolic and therefore isometric to $h_0$.

\item If $h$ is asymptotically cusped of order $k\geq 2$, then there exists a normalized Ricci flow $h(t)$ with bubbling-off, starting from $h$ and defined for all time. Therefore, it is not necessary to modify the initial metric as in \eqref{equ_h_i} or to run the Ricci flow starting from different modified initial data. We obtain 
\begin{equation}\label{equ_rigidity_asymp_cusped}
    \frac{d}{dt}\vol_{h(t)}(M(s))
   = -\int_{M(s)} (R(h(t))+6)\dvol,
\end{equation}
provided that there is no surgery in $M(s)$. 

Suppose that the equality $\vol_{h_0}(M)=\vol_h(M)$ holds, by the inequality in Theorem~\ref{thm_volume_comparison}, it follows that \begin{equation}\label{equ_rigidity_asymp_cusped_2}
    \vol_{h(t)}(M) = \vol_{h_0}(M),
\end{equation}
and no surgery can occur. Otherwise, \eqref{equ_rigidity_asymp_cusped} would yield a metric $h(t)$ that violates the inequality.

Assume that $h$ is not hyperbolic. By the maximum principle, we have $R(h(t))\geq -6$ for $t\geq0$. Moreover, by the strong maximum principle, we see that if for $t>0$, $R(h(t))$ is equal to $-6$ at an interior point, then $R(h(t))\equiv-6$ and $\overset{\circ}{Ric}\equiv 0$, which in turn implies that $h(t)$ would be hyperbolic. Since this contradicts $h$ not being hyperbolic, for a fixed compact set $M(s_0)$ and a closed time interval $[1,2]$, there exists $\delta>0$ so that \begin{equation*}
    R(h(t)|_{M(s_0)})\geq -6+\delta,\quad \forall t\in [1,2].
\end{equation*}
Substituting this into \eqref{equ_rigidity_asymp_cusped}, we have \begin{equation*}
    \vol_{h(2)}(M(s_0))\leq e^{-\delta}\vol_{h(1)}(M(s_0)).
\end{equation*}
On the other hand, we must have \begin{align*}
    &\vol_{h(2)}(M\setminus M(s_0))-\vol_{h(1)}(M\setminus M(s_0))\\
    =&-\int_{1}^2\int_{M\setminus M(s_0)} (R(h(\tau))+6)\dvol\,d\tau\leq 0.
\end{align*}
Hence, it follows that $\vol_{h(2)}(M)<\vol_{h(1)}(M)=\vol_{h_0}(M)$, which contradicts \eqref{equ_rigidity_asymp_cusped_2}. 
Consequently, the metric $h$ is hyperbolic and isometric to $h_0$. 
\end{enumerate}

\section{Applications}\label{section_applications}
Using Theorem~\ref{thm_volume_comparison}, we can generalize certain results from \cite{Agol-Storm-Thurston}, which apply to closed hyperbolic 3-manifolds, to the case of finite volume.
\begin{Cor}
    Let $(M,h)$ be a finite-volume 3-manifold with a smooth metric $h$ such that $R(h)\geq -6$, and the boundary of $M$ is a closed minimal surface. Assume that $DM$, the double of $M$ along its boundary, admits a hyperbolic metric $h_0$. Then \begin{equation*}
        \vol_h(M)\geq \frac12\vol_{h_0}(DM)=\frac{1}{2}v_3\Vert DM\Vert ,
    \end{equation*} 
    where $v_3\Vert DM\Vert $ denotes the simplicial volume of $DM$. 
    
    Furthermore, suppose that $h$ either satisfies $\Vert h-h_0\Vert _{C^2(M)}\leq \epsilon$ for a given constant $\epsilon>0$, or it is asymptotically cusped of order at least two. If the equality holds, then $h$ has constant sectional curvature $-1$ and the boundary of $M$ is totally geodesic with respect to $h$.
\end{Cor}

\begin{proof}
We follow the strategy of \cite{Agol-Storm-Thurston}. We double the manifold $(M,h)$ metrically across its boundary $\Sigma=\partial M$ to form the manifold $DM$, equipped with the piecewise smooth Lipchitz continuous metric obtained by two copies of $h$. We still denote the resulting metric on $DM$ by $h$. As $h$ is not smooth in general, we cannot readily apply Theorem~\ref{thm_volume_comparison} to compare its volume to $\vol_{h_0}(DM)$. Instead, as in \cite[Proposition 4.2]{Agol-Storm-Thurston}, one can modify the metric on a neighbourhood of $\Sigma$ to obtain a sequence of smooth metrics $h_i$ in $DM$ with $R(h_i)\geq -6$ and $h_i\xrightarrow[]{C^0} h$ \emph{globally}. Moreover by following \cite[Theorem 6.1]{Agol-Storm-Thurston} one constructs families of metrics $\lbrace h(t)\rbrace _{0\leq t\leq T}, \lbrace h_i(t)\rbrace _{0\leq t\leq T}$ so that $h(0)=h, h_i(0)=h_i$ and so that for any $t_0>0$ the families $\lbrace h(t)\rbrace_{t_0\leq t\leq T}, \lbrace h(t)\rbrace _{0\leq t\leq T}$ are diffeomorphism-conjugate to a normalized Ricci flows with $R\geq -6$. As $t_0\to 0$, $h(t_0)$ converges uniformly on compact sets to $h$, while as $i\rightarrow+\infty$ we have that $h_i(t)\xrightarrow[]{C^0} h(t)$ for any $0\leq t\leq T$. Applying Theorem~\ref{thm_volume_comparison} to $(DM,h_i(t))$ we obtain
\begin{equation*}
        \vol_{h_i}(M)= \frac{1}{2}\vol_{h_i}(DM) \geq \frac{1}{2}\vol_{h_i(t)}(DM) \geq \frac{1}{2}\vol_{h_0}(DM)=\frac{1}{2}v_3\Vert DM\Vert .
\end{equation*}
Taking $i\rightarrow+\infty$ we get
\begin{equation*}
        \vol_{h}(M)= \frac{1}{2}\vol_{h}(DM) \geq \frac{1}{2}\vol_{h(t)}(DM) \geq \frac{1}{2}\vol_{h_0}(DM)=\frac{1}{2}v_3\Vert DM\Vert .
\end{equation*}
When the equality holds, we must have $\vol_{h(t)}(DM)=\vol_{h_0}(DM)$ for any $0\leq t\leq T$. 
By the proof of Theorem~\ref{thm_volume_comparison} we must have that $R\equiv -6$ for all time. As in Section~\ref{subsection_rigidity} this means that $M$ has constant sectional curvature equal to $-1$, and the gluing between two copies of $M$ produces a smooth metric along the boundary, therefore the boundary is totally geodesic. 
\end{proof}

\begin{Cor}
    Let $(M,h_0)$ be a hyperbolic 3-manifold of finite volume, and let $S$ be an embedded essential surface in $M$. Suppose that $h$ is a metric on $M$ with $R(h)\geq -6$. Then \begin{equation*}
        \vol_h(M)\geq \frac{1}{2}v_3\Vert D(M\backslash\backslash S)\Vert ,
    \end{equation*}
    where $M\backslash\backslash S$ is the Riemannian manifold obtained by taking the path metric completion of $M\setminus S$.
\end{Cor}

\begin{proof}
By Theorem~\ref{thm_volume_comparison},
    \begin{equation*}
        \vol_h(M)\geq \vol_{h_0}(M)\geq \frac{1}{2}v_3\Vert D(M\backslash\backslash S)\Vert ,
    \end{equation*}
    the latter inequality follows by Theorem 9.1 of \cite{Agol-Storm-Thurston}. 
\end{proof}

\bibliographystyle{plain} 
\bibliography{ref}   
\end{document}